\documentclass[12pt]{amsart}
\usepackage{amssymb,latexsym,amsmath,amsthm,mathrsfs,yfonts}
\usepackage{fullpage,enumerate}
\input xy
\xyoption{all}

\newcommand{\CC}{\mathbb C}

\newcommand{\NN}{\mathbb N}

\newcommand{\map}{\rightarrow}

\def\C{\textup{C}}

\def\id{\mathrm{id}}

\def\K{\textup{K}}

\def\sC{\text{$\sigma$-$C^*$}}

\def\lim{{\varprojlim}}
\def\prot{\hat{\otimes}}

\def\dlim{\varinjlim}
\def\ilim{\varprojlim}

\newcommand{\beq}{\begin{eqnarray}}
\newcommand{\beqn}{\begin{eqnarray*}}
\newcommand{\eeq}{\end{eqnarray}}
\newcommand{\eeqn}{\end{eqnarray*}}

\newtheorem*{thm}{Theorem}
\newtheorem*{lem}{Lemma}

\newtheorem*{defn}{Definition}
\newtheorem*{rem}{Remark}

\newtheorem*{ex}{Example}

\begin{document}

\title{Operator algebra quantum groups\\ of universal gauge groups}
\author{Snigdhayan Mahanta}
\email{snigdhayan.mahanta@adelaide.edu.au}
\author{Varghese Mathai}
\email{mathai.varghese@adelaide.edu.au}
\address{Department of Pure Mathematics,
University of Adelaide,
Adelaide, SA 5005, 
Australia}

\thanks{{\em Acknowledgements}. Both authors gratefully acknowledge support under the Australian Research Council's {\em Discovery Projects} funding scheme.}

\keywords{$\sC$-quantum groups, universal gauge groups, operator algebras}

\subjclass[2010]{58B32, 58B34, 46L65}
%\setcounter{tocdepth}{2}
%\tableofcontents

%%%%%%%%%%%%%%%%%%%%%%%%%%%%%%%%%%%%%%%%%%%%%%%%%%%%%%

\begin{abstract}
In this paper, we quantize universal gauge groups such as $SU(\infty)$, in the $\sC$-algebra setting. More precisely, we propose a concise definition of $\sC$-quantum groups and explain the concept here. At the same time, we put this definition in the mathematical context of countably compactly generated groups as well as $C^*$-compact quantum groups.
\end{abstract}

\maketitle

If $H$ is a compact and Hausdorff topological group, then the $C^*$-algebra of all continuous functions $\C(H)$ admits a comultiplication map $\Delta: \C(H)\map \C(H)\prot\C(H)$ arising from the multiplication in $H$. This observation motivated Woronowicz (see, for instance, \cite{W}), amongst others such as Soibelman \cite{S}, to introduce the notion of a {\it $C^*$-compact quantum group} in the setting of operator algebras as a unital $C^*$-algebra with a coassociative comultiplication, satisfying a few other conditions.  If the group $H$ is only locally compact then the situation becomes significantly more difficult. One of the reasons is that the multiplication map $m:H\times H\map H$ is no longer a proper map and one needs to introduce multiplier algebras of $C^*$-algebras to obtain a comultiplication, see for instance, Kustermans-Vaes \cite{KV}. For an excellent and thorough introduction to this theory the readers are referred to, for instance, \cite{KVVVDW}. In the sequel we show that if $H=\dlim_n H_n$ is a countably compactly generated group, i.e., if $H_n\subset H_{n+1}$ are compact and Hausdorff topological groups for all $n\in\NN$ and if $H$ is the direct limit, then a story similar to the compact group case goes through using the general framework of $\sC$-algebras as systematically developed by Phillips \cite{P1,P2}, motivated by some earlier work by Arveson, Mallios, Voiculescu, amongst others. There is a clean formulation of, what we call, {\it $\sC$-quantum groups}, which are noncommutative generalizations of $\C(H)$. Examples of countably compactly generated groups are $U(\infty)=\dlim_n U(n)$, $SU(\infty) =\dlim_n SU(n)$, where $U(n)$ (resp. $SU(n)$) are the unitary (resp. special unitary) groups. They are also known in the physics literature as {\em universal gauge groups}, see Harvey-Moore \cite{HM} and Carey-Mickelsson \cite{CM}. Such spaces are not locally compact and hence the existing literature on quantum groups cannot handle them. Moreover, locally compact groups that are not compact, are also not countably compactly generated. We also discuss in detail the interesting example of the quantum version of the universal special unitary group, $C(SU_q(\infty))$. 

A {\it pro $C^*$-algebra} is an inverse limit of $C^*$-algebras and $*$-homomorphisms, where the inverse limit is constructed inside the category of all topological $*$-algebras and continuous $*$-homomorphisms. For the general theory of topological $*$-algebras one may refer to, for instance, \cite{M}. The underlying topological $*$-algebra of a pro $C^*$-algebra is necessarily complete and Hausdorff. It is not a $C^*$-algebra in general; it would be so if, for instance, the directed set is finite. If the directed set is countable, then the inverse limit is called a {\it $\sC$-algebra}. One can choose a linearly directed cofinal subset inside any countable directed set and the passage to a cofinal subset does not change the inverse limit. Therefore, we shall always identify a $\sC$-algebra $A\cong\ilim_n A_n$, where $n\in\NN$. The inverse limit could have also been constructed inside the category of $C^*$-algebras; however, the two results will not agree. For instance, if $H=\dlim_n H_n$ as above, then the inverse limit $\ilim_n \C(H_n)$ inside the category of topological $*$-algebras is $\C(H)$, whereas that inside the category of $C^*$-algebras is $\C_b(H)$, i.e., the norm bounded functions on $H$. It is known that $\C_b(H)\cong\C(\beta H)$, where $\beta H$ is the Stone--{\v{C}}ech compactification of $H$. Therefore, if one wants to model a space via its algebra of all continuous functions then the former inverse limit is the appropriate one. {\em Henceforth, the inverse limits are always constructed inside the category of topological $*$-algebras}. It is known that any $*$-homomorphism between two pro $C^*$-algebras is automatically continuous, provided the domain is a $\sC$-algebra (see Theorem 5.2. of \cite{P1}). Furthermore, the category of commutative and unital $\sC$-algebras with unital $*$-homomorphisms (automatically continuous) is contravariantly equivalent to the category of countably compactly generated and Hausdorff spaces with continuous maps via the functor $X\mapsto \C(X)$ (see Proposition 5.7. of \cite{P1}). If $A\cong\ilim_n A_n$, $B\cong\ilim_n B_n$ are two $\sC$-algebras, then the minimal tensor product is defined to be $A\prot_{\textup{min}} B = \ilim_n A_n\prot_{\textup{min}} B_n$. {\em Henceforth, $A\prot B$ will always denote the minimal or spatial tensor product between $\sC$-algebras}.

If $H$ is a countably compactly generated and Hausdorff topological group, although the multiplication map $m:H\times H\map H$ is not proper, we get an induced comultiplication map $m^*:\C(H)\map\C(H\times H)\cong\C(H)\prot\C(H)$, which will be coassociative owing to the associativity of $m$. Motivated by the definition of Woronowicz (see also Definition 1 of \cite{KV}), we propose:

\begin{defn}
A unital $\sC$-algebra $A$ is called a $\sC$-quantum group if there is a unital $*$-homomorphism $\Delta: A\map A\prot A$ which satisfies coassociativity, i.e., $(\Delta\prot\id)\Delta = (\id\prot\Delta)\Delta$ and such that the linear spaces $\Delta(A)(A\prot 1)$ and $\Delta(A)(1\prot A)$ are dense in $A\prot A$.
\end{defn}

\begin{lem} 
Let $\{A_n,\theta_n: A_n\map A_{n-1}\}_{n\in\NN}$ be a countable inverse system of $C^*$-algebras and let $B_n\subset A_n$ be dense subsets for all $n$ such that $\theta_n(B_n)\subset B_{n-1}$. Then $\ilim_n B_n$ is a dense subset of the $\sC$-algebra $\ilim_n A_n$.
\end{lem}

\begin{proof}
The assertion follows from the Corollary to Proposition 9 in \S 4-4 of \cite{Bou}.
\end{proof}

\begin{ex}
Obviously, any $C^*$-compact quantum group is a $\sC$-quantum group. Let $\{A_n,\theta_n: A_n\map A_{n-1}\}_{n\in\NN}$ be a countable inverse system of $C^*$-compact quantum groups with $\theta_n$ surjective and unital for all $n$. Furthermore, let us assume that the comultiplication homomorphisms $\Delta_n$ form a morphism of inverse systems of $C^*$-algebras $\{\Delta_n\}:\{A_n\}\map\{A_n\prot A_n\}$. Then $(A,\Delta)=(\ilim_n A_n,\ilim_n\Delta_n)$ is a $\sC$-quantum group. Indeed, the density of the linear spaces $\Delta(A)(A\prot 1)$ and $\Delta(A)(1\prot A)$ inside $A\prot A$ follow from the above Lemma.
\end{ex}

Our next goal is to outline the construction of the quantum universal special unitary group, $C(SU_q(\infty))$.
Recall that for $q\in (0,1)$, the $C^*$-algebra $C(SU_{q}(n))$ is the universal $C^*$-algebra 
generated by $n^{2}+2$ elements $G_n:=\{u^n_{ij}: \, i, j =1, \ldots , n\}\cup\{0,1\}$, which satisfy the following
relations

\begin{equation}\label{relations0}
0^*=0^2=0,\;\;1^*=1^2=1, \;\;01=0=10, \;\;1u^n_{ij} = u^n_{ij}1 = u^n_{ij},\;\; 0u^n_{ij}=u^n_{ij}0=0 \text{  for all $i,j$}
\end{equation}

\begin{equation}\label{relations1}
 \sum_{k=1}^{n}u^n_{ik}(u^{n}_{jk})^*=\delta_{ij}1, \qquad \sum_{k=1}^{n}(u^{n}_{ki})^*u^n_{kj}
=\delta_{ij}1
\end{equation}
\begin{equation}\label{relations2}
\sum_{i_{1}=1}^{n}\sum_{i_{2}=1}^{n} \cdots \sum_{i_{n}=1}^{n}
E_{i_{1}i_{2}\cdots i_{n}}u^n_{j_{1}i_{1}}\cdots u^n_{j_{n}i_{n}} =
E_{j_{1}j_{2}\cdots j_{n}}1
\end{equation}
where
\begin{displaymath}
\begin{array}{lll}
E_{i_{1}i_{2}\cdots i_{n}}&:=& \left\{\begin{array}{lll}
                                 0 &\text{whenever}& i_{1},i_{2},\cdots, i_{n} \text{ are not
distinct;} \\
                                 (-q)^{\ell(i_{1},i_{2},\cdots,i_{n})}.
                                 \end{array} \right. 
\end{array}
\end{displaymath}
Here $\delta_{ij}1=0$ if $i\neq j$, where $0$ denotes the element in the generating set of $C(SU_q(n))$.
Moreover, $\ell(i_1,i_2,\cdots,i_n)$ denotes the number of inversed pairs in the permutation $(i_1,i_2,\cdots,i_n)$. 
The $C^{*}$-algebra $C(SU_{q}(n))$ has a $C^*$-compact quantum group
structure with the comultiplication $\Delta$ given by 
\begin{displaymath}
\Delta(0):=0\otimes 0,\;\; \Delta(1):=1\otimes 1 \text{  and   } \Delta(u^n_{ij}):= \sum_{k} u^n_{ik}\otimes u^n_{kj}.
\end{displaymath} 
\noindent 
There is a surjective $*$-homomorphism $\theta_n:C(SU_{q}(n)) \to C(SU_{q}(n-1))$ defined on the generators by
 \begin{eqnarray*}
 \theta_n(x) &:=& x \text{  if $x=0,1$}\\
 \theta_n(u^n_{ij}) &:=& \left\{\begin{array}{ll}
                        u^{n-1}_{ij}&\text{if}~ 1\leq i,j \leq n-1 ,\\
                        \delta_{ij}1 & \text{otherwise},
                        \end{array} \right.
\end{eqnarray*} such that the following diagram commutes for all $n\geqslant 2$

\beqn \label{comult}
\xymatrix{
C(SU_q(n))\ar[rrr]^{\Delta_n} \ar[d]_{\theta_n} &&& C(SU_q(n))\prot C(SU_q(n)) \ar[d]^{\theta_n\prot\theta_n}\\
C(SU_q(n-1))\ar[rrr]^{\Delta_{n-1}} &&& C(SU_q(n-1))\prot C(SU_q(n-1)).
}
\eeqn One can verify this assertion by a routine calculation on the generators. Consequently, for $n\geqslant 2$ the families $\{C(SU_{q}(n)), \theta_n\}$ and $\{C(SU_q(n))\prot C(SU_q(n)),\theta_n\prot\theta_n\}$ form countable inverse systems of $C^*$-algebras and $\{\Delta_n\}:\{C(SU_q(n))\}{\map} \{C(SU_q(n))\prot C(SU_q(n))\}$ becomes a morphism of inverse systems of $C^*$-algebras. We construct the underlying $\sC$-algebra of the universal quantum gauge group as the inverse limit $$C(SU_q(\infty))=\ilim_n C(SU_{q}(n)).$$ In fact, $C(SU_{q}(\infty))$ is a $\sC$-quantum group, since it is the inverse limit of $C^*$-compact quantum groups, where the comultiplication $\Delta$ on $C(SU_{q}(\infty))$ is defined to be $\Delta =\ilim_n \Delta_n$ (see the Example above).

%An element $x$ of a unital $\sC$-algebra $A$ is called {\it positive} if it is of the form $x=y^*y$ for some $y\in A$. A {\it state} is a positive linear functional $\pi: A\map\CC$, such that $\pi(1)=1$. Let $A'$ denote the linear space of all bounded linear functionals on $A$. Then a state $\phi$ is called a {\it Haar state} if $\phi\rho =\rho\phi = \rho(1)\phi$ for all $\rho\in A'$.
%
%\begin{prop} \label{HaarState}
%A $\sC$-quantum group $(A,\Delta)=(\ilim_n A_n,\ilim_n\Delta_n)$ as in the above Example has a unique Haar state.
%\end{prop}
%
%\begin{proof}
%A result of Van Daele \cite{VD} ensures that there is a unique Haar state $\phi_n$ for each $C^*$-compact quantum group $(A_n,\Delta_n)$. There are induced linear maps $\theta_n^*: A'_{n-1}\map A'_n$, which preserve states. For any $\rho_{n-1}\in A'_{n-1}$, one can verify that $\theta^*_n(\phi_{n-1})\theta^*_n(\rho_{n-1}) =\theta^*_n(\rho_{n-1})\theta^*_n(\phi_{n-1})=\theta^*_n(\rho_{n-1})(1)\theta^*_n(\phi_{n-1})$.
%
%Show that $\theta_n^*(A'_{n-1})$ is dense in $A'_n$, whence $\theta^*_n(\phi_{n-1})$ is a Haar state of $A_n$ and by the uniqueness it must be $\phi_n$.
%
%Show that the algebraic direct limit $\dlim_n A'_n$ is dense inside $A'$.
%
%Now argue by density that the class of $\{\phi_n=\theta^*_n\circ\dots\circ\theta^*_2(\phi_1)\}$ in the direct limit is the unique Haar state of $A$.
%\end{proof}
%

If $G$ is a set of generators and $R$ a set of relations, such that the pair $(G,R)$ is {\it admissible} (see Definition 1.1. of \cite{B}), then one can always construct a universal $C^*$-algebra $C^*(G,R)$. For instance, the universal $C^*$-algebra generated by the set $\{1,x\}$, subject to the relations $\{1^*=1^2=1$, $1x=x1=x$, $x^* x= 1=xx^*\}$, is isomorphic to $\C(S^1)$. The generators and relations of $C(SU_q(n))$ described above are also admissible. 

\begin{rem}
All matrix $C^*$-compact quantum groups considered, for instance, in \cite{W1,W}, such that the relations put a bound on the norm of each generator, are of the form $C^*(G,R)$, where $(G,R)$ is an admissible pair of generators and relations.
\end{rem}

Let $\{(G_i, R_i)\}_{i\in\NN}$ be a countable family of admissible pairs of generators and relations, so that $C^*(G_i, R_i)$ exist for all $i$. Let $F(G)$ denote the associative nonunital complex $*$-algebra (freely) generated by the concatenation of the elements of $G\coprod G^*$ and finite $\CC$-linear combinations thereof, where $\coprod$ denotes disjoint union and $G^*=\{g^*\,|\, g\in G\}$ (formal adjoints). We call a relation in $R$ {\it algebraic} if it is of the form $f=0$ (or can be brought to that form), where $f\in F(G)$. For instance, if $G=\{1,x\}$, then $x^*x=1$ is algebraic, whereas $\|x\|\leqslant 1$ is not. If $(G,R)$ is a pair of generators and relations, then a {\it representation} $\rho$ of $(G,R)$ in a (pro) $C^*$-algebra $B$ is a set map $\rho:G\map B$, such that $\rho(G)$ satisfies the relations $R$ inside $B$. If $(G,R)$ is a {\it weakly admissible} pair of generators and relations (see Definition 1.3.4. of \cite{P2}), then one can construct the universal pro $C^*$-algebra $C^*(G,R)$ (see Proposition 1.3.6. of ibid.). It is known that any combination (even the empty set) of algebraic relations is weakly admissible (see Example 1.3.5.(1) of ibid.).

\noindent
We further make the following hypotheses:

\begin{enumerate}[(a)]
\item \label{Hypa} There are surjective maps $\theta_i: G_i\map G_{i-1}$, so that one may form the inverse limit in the category of sets $G=\ilim_i G_i$, with canonical projection maps $p_i:G\map G_i$. We also require the surjections $\theta_i$ to admit sections $s_{i-1}:G_{i-1}\map G_i$ satisfying $\theta_i\circ s_{i-1}=\id_{G_{i-1}}$, so that we get canonical splittings $\gamma_i: G_i\map G$ satisfying $p_i\circ\gamma_i =\id_{G_i}$. The map $\gamma_i$ sends $g_i\map\{h_j\}$, where

\beq \label{split}
h_j = \begin{cases}
g_i \text{  if $j=i$,}\\
\theta_{i-n+1}\circ\cdots\circ \theta_{i} (g_i) \text{  if $j=i-n$, $n>0$,}\\
s_{i+m-1}\circ\cdots\circ s_i (g_i) \text{  if $j=i+m$, $m>0$.}
\end{cases}
\eeq

%\beq \label{compatibility}
%\xymatrix{
%G=\ilim_i G_i\ar[rr]^{p_i}\ar[rrd]_{p_{i-1}} && G_i \ar[d]^{\theta_i}\\
%&& G_{i-1}.
%}
%\eeq 

\item \label{Hypb} We require that for all $i$ the iterated applications of $\theta_j$'s  and $s_k$'s on $G_i$ satisfy $R_i$ for all $j\leqslant i$ and $k\geqslant i$. 

%$p_{i-n+1}\circ\cdots\circ p_{i} (g_i)$ and $s_{i+m-1}\circ\cdots\circ s_i (g_i)$ satisfy $R_{i}$ for all $i,n,m\geqslant 0$.

%Since the relations are algebraic, each $R_i$ is a subset of $F(G_i)$. The maps $\theta_i$ (resp. $s_{i-1}$) extend uniquely to $*$-homomorphisms $\theta_i:F(G_i)\map F(G_{i-1})$ (resp. $s_{i-1}:F(G_{i-1})\map F(G_i)$. We require that $ \theta_i(R_i) =\langle R_{i-1}\rangle$ and $\langle s_{i-1}(R_{i-1})\rangle \subset\langle R_i\rangle$, where $\langle ?\rangle$ denotes the two-sided ideal generated by $?$. 

\end{enumerate} The surjective maps $\theta_i$ induce surjective $*$-homomorphisms $\theta_i:C^*(G_i,R_i)\map C^*(G_{i-1},R_{i-1})$; consequently, $\{C^*(G_i,R_i),\theta_i\}_{i\in\NN}$ forms a countable inverse system of $C^*$-algebras. We may form the inverse limit $\ilim_i C^*(G_i,R_i)$, which is by construction a $\sC$-algebra. Let $(G,R)$ be a pair of generators and relations, where $G=\ilim_i G_i$ and $R$ denotes the set of relations $\{\gamma_i(G_i) \text{ satisfies } R_i \text{ for all $i$}\}$. A {\it representation} $\rho$ of $(G,R)$ in a (pro) $C^*$-algebra $B$ is a set map $\rho:G\map B$, such that $\rho\circ\gamma_i(G_i)$ satisfies $R_i$ inside $B$ for all $i$. We assume that $(G,R)$ is a weakly admissible pair, so that one can construct the universal pro $C^*$-algebra $C^*(G,R)$.

\begin{thm}\label{thm:inverselim}
There is an isomorphism of pro $C^*$-algebras $C^*(G,R)\cong \ilim_i C^*(G_i,R_i)$.
\end{thm}

\begin{proof}
It suffices to show that $\ilim_i C^*(G_i,R_i)$ is a universal representation of $(G,R)$, i.e., there is a map $\iota:G\map\ilim_i C^*(G_i,R_i)$ such that $\iota\circ\gamma_i(G_i)$ satisfies $R_i$ inside $\ilim_i C^*(G_i,R_i)$ for all $i$ and given any representation $\rho$ of the pair $(G,R)$ in a pro $C^*$-algebra $B$, there is a unique continuous $*$-homomorphism $\kappa: \ilim_i C^*(G_i,R_i)\map B$ making the following diagram commute:

\beqn
\xymatrix{
G=\ilim_i G_i\ar[rrd]_{\rho}\ar^{\iota}[rr] &&\ilim_i C^*(G_i,R_i) \ar[d]^{\kappa}\\
&& B.
}
\eeqn 
The map $\iota: G\map\ilim_i C^*(G_i,R_i)$ is defined as $g\mapsto \{p_i(g)\}$, which is a representation of $(G,R)$ due to the Hypothesis \eqref{Hypb} above. The construction of the universal pro $C^*$-algebra $C^*(G,R)$ (resp. $C^*$-algebra $C^*(G_i,R_i)$) is defined via a certain Hausdorff completion of $F(G)$ (resp. $F(G_i)$) with respect to representations in pro $C^*$-algebras (resp. $C^*$-algebras) satisfying $R$ (resp. $R_i$). The surjective maps $\theta_i$ induce $*$-homomorphisms $\theta_i: F(G_i)\map F(G_{i-1})$, whence we may construct  the $*$-algebra $\ilim_i F(G_i)$ (purely algebraic inverse limit). By the above Lemma it suffices to define $\kappa$ on coherent sequences of the form $\{w_i\}\in\ilim_i F(G_i)$, which then extends uniquely to a $*$-homomorphism on the entire $\ilim_i C^*(G_i,R_i)$. Thanks to the maps $\rho\circ\gamma_i:G_i\map B$, $\rho$ extends uniquely to a $*$-homomorphism $\ilim_i F(G_i)\map B$. Now there is a unique choice for $\kappa(\{w_i\})$ forced by the compatibility requirement, i.e., $\kappa(\{w_i\})=\rho(\{w_i\})$. By construction $\kappa$ is a $*$-homomorphism and it is automatically continuous, since $\ilim_i C^*(G_i,R_i)$ is a $\sC$-algebra. 
\end{proof}

In the example of $C(SU_q(\infty))$, one could try to define the section maps $s_{n-1}: G_{n-1}\map G_n$ as $$ 0\mapsto 0,\;\; 1\mapsto 1,\;\; u^{n-1}_{ij}\mapsto u^n_{ij}.$$ But the Hypothesis \eqref{Hypb} will not be satisfied and hence the above Theorem is unfortunately not applicable. However, the Theorem could be of independent interest as it can be applied to inverse systems, where the structure $*$-homomorphisms admit sections (also $*$-homomorphisms).

\noindent
Let $G_n:=\{w^n_{ij}: \, i, j =1, \ldots , n\}\cup\{0,1\}$ be a set of generators satisfying the relations $R_n$ $$0^*=0^2=0,\;\;1^*=1^2=1, \;\;01=0=10, \;\;1w^n_{ij} = w^n_{ij}1 = w^n_{ij},\;\; 0w^n_{ij}=w^n_{ij}0=0,\;\, \|w^n_{ij}\|\leqslant 1$$ for all $i,j$. The pair $(G_n,R_n)$ is an admissible pair for all $n$, so that there is a universal $C^*$-algebra $C^*(G_n,R_n)$. There are surjective maps $\theta_n:G_n\map G_{n-1}$ given by 

 \begin{eqnarray*}
 \theta_n(x) &:=& x \text{  if $x=0,1$}\\
 \theta_n(w^n_{ij}) &:=& \left\{\begin{array}{ll}
                        w^{n-1}_{ij}&\text{if}~ 1\leq i,j \leq n-1 ,\\
                        \delta_{ij}1 & \text{otherwise}.
                        \end{array} \right.
\end{eqnarray*} making $\{C^*(G_n,R_n),\theta_n\}$ an inverse system of $C^*$-algebras and surjective $*$-homomorphisms. There are obvious sections $s_{n-1}: G_{n-1}\map G_n$ sending $0\mapsto 0$, $1\mapsto 1$ and $w^{n-1}_{ij}\mapsto w^n_{ij}$ giving rise to maps $\gamma_n:G_n\map G=\ilim_n G_n$ as described above (see Equation \eqref{split}). There are surjective $*$-homomorphisms $\pi_n: C^*(G_n,R_n)\map C(SU_q(n))$ for all $n\geqslant 2$ given on the generators by $\pi_n(x)=x$ for $x=0,1$ and $\pi_n(w^n_{ij})=u^n_{ij}$, which produce a morphism of inverse systems $\{\pi_n\}:\{C^*(G_n,R_n)\}\map\{C(SU_q(n))\}$. Indeed, it follows from the Relations \eqref{relations0}, \eqref{relations1} and \eqref{relations2} that the norms of the generators of $C(SU_q(n))$ do not exceed $1$ in any representation. Consequently, there is a surjective $*$-homomorphism of $\sC$-algebras $$\ilim_n\pi_n:\ilim_n C^*(G_n,R_n)\map C(SU_q(\infty)).$$
However, the authors cannot provide a good description of the kernel at the moment. Let us set $G=\ilim_n G_n$ and let $R$ denote the set of relations $\{\gamma_n(G_n) \text{ satisfies $R_n$ for all $n$}\}.$ Note that $\|x\|\leqslant 1$ viewed as a relation for a representation in a pro $C^*$-algebra $B$ means that $p(x)\leqslant 1$ for all $C^*$-seminorms $p$ on $B$. The family of pairs $(G_n,R_n)$ satisfy Hypotheses \eqref{Hypb} and the pair $(G,R)$ is weakly admissible (see Examply 1.3.5.(2) of \cite{P2}), so that the above Theorem applies, i.e., $\ilim_n C^*(G_n, R_n)\cong C^*(G,R)$. As a corollary, we deduce that the elements of $(\ilim_n\pi_n)(G)$ provide explicit generators of $C(SU_q(\infty))$.

%\comment{Explain the relations properly.}

%\begin{cor}
%Let $G=\ilim_n G_n$ with splittings $\gamma_n: G_n\map G$ (cf. Equation \eqref{split}), and $R$ denote the set of relations $\{\gamma_n(G_n) \text{ satisfies Equations \eqref{relations0}, \eqref{relations1} and  \eqref{relations2} for all $n\geqslant 2$}\}$. Then $$C(SU_{q}(\infty)) = C^*(G, R).$$
%\end{cor}

%---------------bibliography-------------------------------------------

\bibliographystyle{abbrv}

\bibliography{/Users/smahanta/Professional/math/MasterBib/bibliography}

\vspace{5mm}
\noindent
\address{}

\end{document}